\documentclass[a4paper, reqno, 11pt]{article}
\usepackage{amsmath,epsfig, calc, amsthm}
\usepackage{amssymb, color}
\usepackage{bbm}
\usepackage{enumitem} 
\usepackage[numbers]{natbib}
\newtheorem{theorem}{Theorem}[section]
\newtheorem{lemma}[theorem]{Lemma}
\newtheorem{prop}[theorem] {Proposition}

\theoremstyle{definition}
\newtheorem{example}[theorem] {Example}
\newtheorem{remark}[theorem] {Remark}

\newcommand{\A}{{\mathcal {A}}}

\def\phi{\varphi }

\newcommand{\R}     {\mathbb{R}}
\newcommand{\Z}     {\mathbb{Z}}
\newcommand{\N}     {\mathbb{N}}
\renewcommand{\P}   {\mathbb{P}}
\newcommand{\E}     {\mathbb{E}}

\renewcommand{\d}   {\operatorname{d}\!}
\newcommand{\la}{\lambda}
\newcommand{\ra}{\rightarrow}
\newcommand{\cL}{\mathcal{L}}

\newcommand{\1}{\mathbbm{1}} %{1\hspace{-0.098cm}\mathrm{l}}

\newcommand{\ssup}[1] {{\scriptscriptstyle{({#1}})}}

\newcommand{\p}{\mathbb{P}}

\newcommand{\cA}{\mathcal{A}}
\newcommand{\cF}{\mathcal{F}}
\newcommand{\cG}{\mathcal{G}}

\allowdisplaybreaks[1] 
\begin{document}

\begin{center}
{\Large \bf Fluctuations in a general preferential attachment\\[2mm] model via Stein's method
}\\[5mm]

\vspace{0.7cm}
\textsc{Carina Betken\footnote{Universit\"at Osnabr\"uck, Institut f\"ur Mathematik, D-49076 Osnabr\"uck, Germany, {\tt carina.betken@uni-osnabrueck.de}, {\tt hanna.doering@uni-osnabrueck.de}.}, Hanna D\"oring$^1$}
and \textsc{Marcel Ortgiese\footnote{Department of Mathematical Sciences, University of Bath, Claverton Down, Bath, BA2 7AY,
United Kingdom, {\tt m.ortgiese@bath.ac.uk}.}
} 
\\[0.8cm]
{\small   \today} 
\end{center}

\vspace{0.3cm}

\begin{abstract}
  \noindent 
We consider a general preferential attachment model, where the probability that a newly arriving vertex  connects to an older vertex  is proportional to a sublinear function  of the indegree of the older vertex at that time. It is well known that  the distribution of a uniformly chosen vertex converges to a limiting distribution. Depending on the parameters, this model can show power law, but also stretched exponential behaviour. Using Stein's method  we provide rates of convergence for the total variation distance. 
Our proof uses the fact that the limiting distribution is the stationary distribution of a Markov chain together with the generator method of Barbour.

  %60B12 limit theorems for vector-valued random variables (infinite-dimensional case) 
  % 60C05 (1973-now) Combinatorial probability 
  %  60F05 Central limit and other weak theorems 
\end{abstract}

\section{Introduction and main results}

\subsection{Introduction}

Preferential attachment models, as popularized by Barab\'asi and Albert~\cite{Barabasi1999}, rely only on a simple mechanism to explain the occurrence of power-law degree distributions often observed in real-world networks: the network is modelled as a growing sequence of random graphs, where the probability of connecting a newly incoming vertex to an old vertex is proportional to its degree.  Here we study a  general preferential attachment model, based 
on the model  introduced by Dereich and M\"orters \cite{Dereich2009}. 
The connection probabilities are given by a general (sublinear) function of the old degree and this framework is sufficiently general to lead to typical degree distributions that can be power-laws, but also include stretched exponential distributions.
Our main result gives rates of convergence for the total variation distance of the degree of a uniformly chosen vertex and the limiting degree distribution.

For the model introduced in~\cite{Dereich2009}
we start with $\cG_1$ being one vertex with no edges. In each step we add one vertex.  The graph $\cG_n$ consists of $n$ vertices, but loops or multiple edges are not allowed.  The new vertex connects to each of the former vertices $j$ independently  with probability
\[ \p( n+1 \mbox{ connects to } j \, |\, \cG_n )  = \frac{f(\deg^-_n(j))}{n} , \]
where $\deg_n^-(j)$ is the indegree of vertex $j$ at time $n$ and the function $f : \N_0 \rightarrow (0,\infty)$ is such that $f(n) \leq n+1$. If $f$ is a linear function in $k$, we refer to the model as linear preferential attachment. Here the connections to old vertices are sampled independently, so that the outdegree is random, however our framework also allows for models with fixed outdegree (see below) that are closer to the original mathematical formulation of the Barab\'asi-Albert model due to~\cite{Bollobas2001a}.

Here, we analyse the rates of convergence in total variation of the indegree of a random uniformly chosen vertex, resp.\ the outdegree of vertex $ n $, with Stein's method. Let $ W_n $ denote the indegree of a uniformly chosen vertex at time $ n $ in the model described above.
Then, our main theorem shows under fairly weak conditions on $f$ (essentially that eventually $f(k) \leq k$): if $W$ has the asymptotic degree distribution identified in~\cite{Dereich2009} (see also~\eqref{limitingdistribution} below), then for all $n \geq 2$
\begin{equation}\label{eq:result}
d_{\rm TV}(W_n,W)\leq C ~ \frac{\log(n)}{n}. 
\end{equation}
We also give a weaker result if $f$ is of the form $f(k)  = k+ \gamma$ for some $ \gamma \in (0,1) $. Moreover, 
for the model with random outdegree, \cite{Dereich2009}, we can also obtain rate of convergence results for the outdegree
when compared with  a Poisson distribution.

For linear $f$, a result analogous to ours was shown in~\cite{F09}, \cite{PRR13} and \cite{Ross2013}.
In the thesis \cite{F09} various versions of the preferential attachment model  (including random and fixed outdegree) were studied using couplings and Stein's method. 
However, in her work only the former was used successfully to derive rates of convergence, which  in the best case are of the form~\eqref{eq:result}.
For the linear model with fixed outdegree and where connections are made with probability proportional to the degree plus a constant $\delta$, 
there is an alternative proof of~\eqref{eq:result}
first in~\cite{PRR13} (for $\delta =0$) and then in~\cite{Ross2013}. These proofs use Stein's method, but crucially rely on the fact that the limiting distribution can be represented as a mixture of a geometric (resp.\ negative binomially distributed for $\delta \neq 0$) random variable.

The first mathematical papers on preferential attachment models showed 
the proportion of vertices with a particular degree is close to a power law
including error bounds, see \cite{Bollobas2001a} for the first mathematical paper, \cite{Brightwell2012} for an approach that also works for large degrees and~\cite{Hofstad2017} for an excellent overview.
These results are often proved by first showing concentration of the empirical degree distribution and then showing the expected degree sequence converges. 
However, in these cases the bounds on the expected degree sequence are not strong enough to give~\eqref{eq:result}, see also the discussion in~\cite{Ross2013}.

Other questions that have been studied include error bounds for the degree of a fixed vertex.
Pek\"oz, R\"ollin and Ross \cite{PRR13} applied Stein's method successfully to prove a rate of convergence in Kolmogorov distance for the indegree distribution of any fixed vertex to a power law distribution by comparing it to a mixed negative binomial distribution, whereas in \cite{PRR12} they provide rates of convergence in total variation of a random vertex in a uniform attachment graph by applying Stein's method to the geometric distribution. In \cite{PRR17}
 the authors prove a rate of convergence in the multidimensional case for the joint degree distribution. These articles consider models where every vertex has fixed outdegree.

Our results are based on a new application of Stein's method.
Compared to previous results our methods apply to a generalisation of preferential attachment models having very different limit laws, including power law, but also stretched exponential, tail behaviour. This is remarkable, considering that Stein's method (at least in a first step) requires the characterization of the limiting distribution in terms of the Stein operator. 
Secondly, our method is  direct in that it gives  a Stein operator for the limiting distribution (unlike in~\cite{Ross2013} where it is exploited that the limiting distribution can be written as the mixture of  a negative binomial distribution). 
Our proof relies on the fact that the limiting distribution is the invariant distribution of a continuous-time Markov chain and therefore it allows to use Barbour's generator method~\cite{Barbour1988} to take care of the first part of Stein's method.
Finally, our method is  robust in that it only depends on the marginal distribution of a new vertex being connected to a particular old vertex, and so allows to treat models with  fixed and random outdegree at the same time.

Throughout we will use the following  notation: $\N = \{1, 2,\ldots\}$ denotes the natural numbers, $\N_0 = \N \cup \{0\}$ and $[n] := \{1, \ldots, n\}$.
For any two probability measures  $\mu,\nu$ on $(E, \cF)$ we will denote their total variation distance by
\[ d_{\rm TV} (\mu,\nu) = \sup\{ |\mu(A) - \nu(A)|\, : \, A\in \cF \}. \]
Moreover, for any ranom variables $X,Y$ taking values in $E$, we write
$d_{\rm TV}(X,Y) := d_{\rm TV} (\cL (X), \cL(Y))$.

\subsection{Main results}

Our methods do not rely on the fine details of the model  under consideration, so that we state our assumptions first in all generality and then highlight some of the models that fit in this class.

{\bf Assumptions (A)}. Fix $d_0 \in \N_0$ and let $f : \N_0 \rightarrow (0,\infty)$ such that $f(n) \leq \max\{ n+1 - d_0,1\}$. We assume  that $(\cG_n)_{n \geq 1}$ is a sequence of directed random graphs 
with vertex set $[n] := \{1, \ldots,n\}$. 
The initial graph $\cG_1$ consists of vertex $1$ and $d_0$ (directed) self-edges. 
For any $n\geq 1$, at time $n+1$ we add vertex $n+1$ to the vertex set and insert at most one directed edge from $n+1$
to $j$ such that
\begin{equation}\label{eq:marginal} \p( n+1 \mbox{ connects to } j \, |\, \cG_n )  = \frac{f(\deg^-_n(j))}{n} . \end{equation}
Here $\deg^-_n(j)$ denotes the indegree of vertex $j$ after the $n$th vertex has been inserted.

Note that we have by construction that $\deg_n^-(j) \leq d_0 +  n - 1$, so that by the condition on $f$ the right hand side of~\eqref{eq:marginal} is indeed $\leq 1$.

These assumptions do not completely specify the model: they allow for deterministic as well as random outdegree and also only impose conditions on the marginal probabilities of connecting $n+1$ to $j$. In particular, the following models are included.

\begin{example}[Preferential attachment with random outdegree, see~\cite{Dereich2009}]\label{ex:1} \ \\
Suppose $f(n) \leq n+1$ for all $n \in \N_0$.
 We start with the graph $\mathcal{G}_1$ consisting of one vertex and no edges. 
		At time $ n+1 $, we add vertex $n+1$ to the vertex set and
 independently for each $ k \in [n]$ add a directed edge from $ n+1 $ to $ k $ with probability
			$$
			\frac{f(\deg^-_{n}(k))}{n}.
			$$
\end{example}

\begin{example}[Preferential attachment with fixed outdegree]\label{ex:2}
Start with $\cG_1$ consisting of vertex $1$ and a (directed) self-loop. At time $n+1$ insert vertex $n+1$ and
connect it to exactly one previous vertex $j \in [n]$ with probability 
\[ \frac{\deg_n(j) + \delta}{n(2 + \delta)}, \]
where $\deg_n(j)$ denotes the total degree of vertex $j$ at time  $n$ and $\delta  >-1$ is a parameter.
Noticing that $\deg_n(j) = \deg_n^- (j) +1$, this fits into our framework with $f(k) = \frac{k + (1+\delta)}{2 +\delta}$ and $d_0 = 1$.
This model is almost the model proposed in~\cite{Bollobas2001a} (where however $\delta = 0$ and we do not allow for self-loops of vertex $n+1$) and it is very closely related to what is referred to as model (b) in~\cite[Chapter 8.2]{Hofstad2017}.
\end{example}

\begin{example}[Spatial preferential attachment model]
In \cite{Aiello08}, the authors introduce the following spatial random graph model. Let $ S $ be the unit hypercube in $ \R^m $. The initial graph consists of  vertex $1$ that is placed uniformly at random into $S$ and no edges.
For each vertex $i$ we define \textit{the sphere of influence $ S(i,n) $ of i} as the ball (in the torus metric induced by the Euclidean metric) that has volume $ \frac{A_1 deg_n^i(i)+A_2}{n} $ centered at the position of $ i $, where
$A_1, A_2 \geq 0$. Fix a parameter $p \in [0,1]$. Then, 
at time $ n+1 $, we insert vertex $n+1$ at a position that is chosen uniformly at random in $S$.
Now, independently for each vertex $j$ such that the position of $n+1$ is in $S(j,n)$ insert an edge from $n+1$ to 
$j$ with probability $p$. 
In particular, we get
\[ \p(n+1 \text{ connects to } j \, | \,  \mathcal{G}_n)=p\cdot \frac{A_1 deg_n^-(j)+A_2}{n}. \]
Thus, this model fits into our framework if we choose $ f(k)=pA_1 k+ p A_2 $ and where we assume that 
the constants are chosen such that $pA_1, p A_2 \leq 1$.
\end{example}

The limiting distribution of the asymptotic distribution of the indegree of a uniform vertex in the above examples is known (see e.g.\ \cite[Thm. 1.1]{Dereich2009}, \cite[Thm.~1.1]{Aiello08}) and 
it is given as $\mu = (\mu_k)_{k \in \N_0}$ with
\begin{equation}
\mu_k=\frac{1}{1+f(k)}\prod\limits_{i=0}^{k-1}\frac{f(i)}{1+f(i)},\quad k \in \N_0. \label{limitingdistribution}
\end{equation}
This  is a probability distribution (see e.g.~Lemma~\ref{le:Stein_op} below) and moreover, by varying $f$, the limiting distribution  can exhibit very different tail behaviour. 
As discussed in Examples 1.3 and 1.4 in \cite{Dereich2009} we obtain for  $ f(k)=\gamma k +\beta  $
\[
\mu_k \sim \frac{\Gamma(\frac{\beta +1}{\gamma})}{\gamma \Gamma(\frac{\beta}{\gamma})} k^{-(1+\frac{1}{\gamma})}, \quad \text{for } k\rightarrow \infty . 
\]
Therefore, our framework allows for models with power-law distribution with tail exponent $1 + 1/\gamma \in [2,\infty)$.
Furthermore, if $ f(k)\sim \gamma k^{\alpha} $ with $ 0<\alpha<1 $, $ \gamma >0 $, then
\[
\log \mu_k \sim \frac{1}{\gamma(1-\alpha)}k^{1-\alpha}, \quad \mbox{for } k \ra \infty,
\]
so that we obtain a limiting distribution with stretched exponential tails.

We can now state our main theorem.

\begin{theorem}\label{mainthm2}
Let $ W_n $ denote the indegree of a uniformly chosen vertex at time $ n $ in a preferential attachment model satisfying assumptions (A). Suppose further that  there exists $k_\ast\in \N_0$ such that $f(k) > k$ for all $k < k_\ast$ and $f(k) \leq k$ for all $k \geq k_\ast$.
Then, there exists a constant $C >0$ such that for all $n \geq 2$ 
	\begin{equation}
	 d_{\rm TV}(W_n,W)\leq C~ \frac{\log(n)}{n},
	\end{equation}
	where $ W \sim \mu $ and $ \mu $ as in (\ref{limitingdistribution}). 
\end{theorem}

The condition that  there exists $k_\ast\in \N_0$ such that $f(k) > k$ for all $k < k_\ast$ and $f(k) \leq k$ for all $k \geq k_\ast$ is for example fulfilled for all sublinear models such that $\max_k \Delta f(k) < 1$, which is a popular condition in the setting of Example~\ref{ex:1}, see e.g.~\cite{Dereich2013}.

The next result gives a weaker result in the regime, when $f(k) \in [k,k+\gamma]$ for all $k$. This is not surprising since for example in the case $ f(k)=k +\gamma $, the distribution has power law exponent $2$  and does no longer have a finite mean.

\begin{theorem}\label{mainthm}
	Let $ W_n $ denote the indegree of uniformly chosen vertex at time $ n $ in a model satisfying assumptions (A). Suppose further that $ f(k) \in [k, k+\gamma ]$ for all $k \in \N_0$ for some $ \gamma \in (0,1) $. Then, there exists a constant $C > 0$ such that for all $n \geq 1$,
	\begin{equation*}
	 d_{\rm TV}(W_n,W)\leq C ~ n^{-(1- \gamma)},
	\end{equation*}
	where $ W \sim \mu $ and $ \mu $ as in (\ref{limitingdistribution}). 
	\end{theorem}

	Theorem \ref{mainthm2} as well as Theorem \ref{mainthm} deal with the indegree of a uniformly chosen vertex.
However, for the model described in Example~\ref{ex:1} it also makes sense to ask about the
distribution of the random outdegree. \cite{Dereich2009} show that the outdegree is approximately Poisson, the next theorem gives an 
error bound on this approximation.

\begin{theorem}\label{mainthm3}
	Let $ D_n $ denote the outdegree of vertex  $ n $ in the model described in Example~\ref{ex:1} and suppose for some $\gamma \in (0,1)$, we have $f(k)\leq \gamma k+1$ for all $k \in \N_0$. Then there exists $C> 0$ such that
	\begin{equation*}
	d_{\rm TV}(D_n,Po(\lambda_n))\leq C  \, 
	\begin{cases}
	\frac{1}{n+1},&\text{ for } 0<\gamma < \frac{1}{2},\\
	\frac{\log(n)}{n}, &\text{ for } \gamma = \frac{1}{2},\\
	n^{-2(1-\gamma)},&\text{ for }\frac{1}{2}< \gamma <1 ,\\
	\end{cases}
	\end{equation*}
	where $ Po(\lambda_n) $ denotes the Poisson distribution with parameter $ \lambda_n=\E\left[f(W_{n-1})\right]$ and  $W_{n-1}$ has the distribution of the indegree of a uniformly chosen vertex at time $n-1$. Moreover, $\la_n \ra \la := \E[f(W)]$, 
	where $W \sim \mu$  as in~\eqref{limitingdistribution}. 
Finally, if $f(k) = \gamma k + \beta$ for $\gamma\in (0,1), \beta \in [0,1]$, then 
	\[
	\left|\lambda_n-\E\left[f(W)\right]\right| \leq n^{-1+\gamma}.
	\]
\end{theorem}

\begin{remark} The assumptions on the initial graph as formulated in the Assumptions (A) are fairly restrictive. For example, 
the model (b) described in~\cite[Chapter 8.2]{Hofstad2017} does not fit exactly, as in that model we start with a graph of 
two vertices connected by two edges.
Indeed, if we consider the model with fixed outdegree and change the initial conditions, then the normalization might not be 
an exact multiple of $n$. However, our proofs will also go through, if we change the assumption on the marginal connection probabilities to 
\[ \p( n+1 \mbox{ connects to } j \, | \, \cG_n ) = \frac{ f(\deg_n^-(j))}{n + \alpha} \, , \]
for some $\alpha \geq 0$. With this modification and a corresponding specification of the initial conditions, we can 
also deal with models with arbitrary starting configurations in the setting of Example~\ref{ex:2}, such as model (b) in~\cite{Hofstad2017}.
Another example that would fit into this set-up is the model $ \bar{I} $ described in \cite{F09}, 
where $f$ is linear and connections are made independently to each previous vertex, but the normalization is $n + \alpha$.
In particular, using the coupling arguments in~\cite{F09}, our results also transfer to the other models described there.
In the following,  in order to keep the notation simple, we restrict ourselves to the case $\alpha = 0$.
\end{remark}

\begin{remark} Our method does not cover directly the Barab\'asi-Albert model in the form introduced in~\cite{Bollobas2001a} 
(and so neither the generalisation in~\cite{Hofstad2017}), because they include the possibility of a self-loop for the incoming vertex.
It is possible to couple these models with model (b) (from the previous remark).
Unfortunately this coupling gives an error of order $ \frac{\log(n)^3}{n}. $
Therefore, we cannot recover the results of
~\cite{PRR12} and~\cite{Ross2013}, who obtain the same error rates as in our Theorem~\ref{mainthm2} for these models.

Other generalisations that we cannot deal with are the models with fixed outdegree where the probability to connect to an old vertex is proportional to a non-linear function as in~\cite{RTV07}, which leads to a random normalization. Since this problem arises in a wide class of  modifications of preferential attachment models, dealing with this problem is an interesting topic of future research.
\end{remark}

\subsection{Stein's method and overview of proof}\label{sec:proof_strategy}

Stein's method is a well-known tool for deriving rates of convergence of a random variable to a known target distribution. First developed for the normal distribution in \cite{Stein}, it has been adapted to several other cases of target distributions, including Poisson in \cite{Ch75} but also many others.
For a detailed overview of Stein's method see for example \cite{Ross2011}.

 The method can be divided into three main steps: first one has to find a characterizing operator $ \mathcal{A} $ of the target distribution $ \mu $, in the sense that for all functions $ g $ in the domain of $ \mathcal{A} $ and any random variable $W$,
\begin{equation}\label{eq:stein1}
\E\left[\mathcal{A}g(W)\right]=0 \Leftrightarrow  W\sim \mu.
\end{equation} The second step now consists of 
finding a solution $g := g_h $  to the  Stein-equation
\begin{equation}
h(k)-\int h \d \mu=\mathcal{A}g(k)\label{Equation}
\end{equation}
for each $ h $ in a convergence-determining class of functions $ \mathcal{F} $. This then yields for any random variable $X$ and a random variable $W \sim \mu$ that
\begin{equation}
\d_{\mathcal{F}}(W,X):=\sup_{h\in \mathcal{F}}\left|\E\left[h(W)\right]-\E\left[h(X)\right]\right|=\sup_{h \in \mathcal{F}}\E\left[\mathcal{A}g_h(X)\right].\label{Distance}
\end{equation}
Depending on $ \mathcal{F} $ we get different distances on probability measures. Here we are interested in the total variation distance so that we choose  $ \mathcal{F}=\{\mathbbm{1}_A(\cdot), A \in\text{Borel}(\R)\} $. 
In the last step one has to bound the right-hand side of (\ref{Distance}) to get bounds on the distributional distance of $ W $ and $ X $.

In our case, in order to deal with steps  1 and 2, we will follow the generator approach introduced by Barbour \cite{Barbour1988}.
 This approach applies when $\cA$ is the (infinitesimal) generator of a continuous-time Markov chain $(Z_t)_{t \geq 0}$. In particular, we have
that~\eqref{eq:stein1} holds if and only if $\mu$ is the stationary distribution of the Markov chain $(Z_t)_{t \geq 0}$. 
In Lemma~\ref{le:Stein_op}, we will see that the Stein operator for our target distribution $\mu$ in~\eqref{limitingdistribution} is given as
\begin{equation}\label{operator}
\A g(k)=  f(k) \Delta g(k)  + g(0) - g(k) , 
\end{equation}
where $g : \N_0 \rightarrow \R$, $\Delta g(k) := g(k+1) - g(k)$ and $ g $ such that $ \E\left[g(Y)\right]< \infty $ for $ Y \sim \mu $.
So the corresponding Markov process is a jump process $(Z_t)_{\geq 0}$ on $\Z_0$ that jumps from state $k$ to $k+1$ at rate $f(k)$ and from any state at rate $1$ back to the state $0$.

Applying standard results for Markov chains also allows to take care of step 2 and we will show in Lemma~\ref{le:stein_sol} that $g_h$ is given as 
\begin{equation}\label{eq:stein_sol} g_h(k)= -\int_0^\infty \Bigl( \E_k h(Z_t)-\int h \d\mu\Bigr)  \d t ,  \end{equation}
for any $h \in \mathcal{F}$.

In the following we will explain the main parts of step $3$, i.e.\ how to derive the claimed error bounds from the right hand side of~\eqref{Distance}.
The first part is independent of preferential attachment models and we will show
in Section~\ref{sectionbounds}  the following smoothness estimate for the Stein solution: for any $g_A := g_{\1_A}$ and $A \subset \N_0$, we have that
\begin{equation}\label{eq:stein_smooth} f(k) \Delta g_A(k) \leq 1, \quad \mbox{for all } k \in \N_0. \end{equation}
In the second part, we use the following dynamic way of generating a uniform random variable on $[n]$ (cf. \cite{F09}): let $J_n$ be a Markov chain with $J_1=1$ and such that
\[ \p(J_{n+1} = J_n \, | \, J_n) = \frac{n}{n+1} \quad \mbox{ and } \quad  \p(J_{n+1} = n+1 \, | \, J_n ) = \frac{1}{n+1}. \]
Then, we have that $J_n$ is uniformly distributed on $[n]$ for every $n$ (cf. also Lemma~\ref{LemmaMC}). In particular, we know that we can generate the indegree of a uniform vertex as $X_n := \deg^-_n(J_n)$ and moreover $(X_n)$ turns out to be a Markov chain. If we additionally assume that $d_0 = 0$ in the Assumptions (A), then as a first step we use this Markov structure to show in Lemma~\ref{le:recursion} that  
\begin{equation} \E\left[\mathcal{A}g_{A}(X_{n+1})\right]
	=\frac{1}{n+1}\Big( \sum_{\ell=1}^{n} \sum_{k = 0}^{\ell -1}
	\Delta v_A(k) h(k,\ell) +v_A(0)\Big) \label{Triplesum2}, 
	\end{equation}
	where
	\[ v_A(k) := f(k) \Delta g_{A}(k) \mbox{ and }
	h(k,\ell) := f(k)\P(X_{\ell}= k)- \P(X_{\ell}\geq k+1) . \]
Then, since we have the smoothness estimate~\eqref{eq:stein_smooth}, as a last step it remains to analyse $h(k,\ell)$ and show that these terms are small.
The corresponding analysis is carried out in Proposition~\ref{prop_difference}. We first show inductively, that under the conditions of Theorems~\ref{mainthm2} and~\ref{mainthm}, for fixed $\ell$ the functions $k \mapsto h(k,\ell)$
are first increasing and then decreasing, which allows ultimately to deal with the inner sum over $k$. Then, finally we show that for suitable constants $C>0$, $k\leq \ell-1$,
\[ h(k,\ell) \leq \left\{ \begin{array}{ll} C \, \ell^{-1}   & 
\mbox{ under the assumptions of Thm.~\ref{mainthm2}},\\[3mm]
C\, \ell^{- (1-\gamma) }  & \mbox{ under the assumptions of Thm.~\ref{mainthm}}.
\end{array} \right.
\]
These bounds then lead via~\eqref{Triplesum2} to the error bounds of $\log(n)/n$ and $n^{-(1-\gamma)}$ in Theorems~\ref{mainthm2} and~\ref{mainthm}. The case $d_0 > 0$ will be shown using an easy coupling argument.

\section{Proofs: Stein for the indegreee distribution}

We will follow the strategy outlined in Section~\ref{sec:proof_strategy}.

\subsection{Deriving the Stein operator}\label{sec:stein}

In this section, we will deal with the Stein operator for $\mu$ defined in~\eqref{limitingdistribution} and show how to solve the Stein equation \eqref{Equation}. This corresponds to steps 1 and 2 of the general strategy outlined in Section~\ref{sec:proof_strategy}.

\begin{lemma}\label{le:Stein_op} Let $\mu$ be given by~\eqref{limitingdistribution}. Then, $\mu$ is a probability distribution  and any $\N_0$-valued random variable $W$ satisfies $W \sim \mu$ if and only if
\begin{equation}\label{eq:0205-1} \E [ \mathcal{A} g(W) ] = 0 , \end{equation}
for all $g:\N_0\rightarrow \R$ such that $ \E[g(W)]<\infty $, 
where 
\[ \cA g(k) :=  f(k) \bigl( g(k+1)-g(k)\bigr) + g(0) - g(k). \]
\end{lemma}

\begin{proof}
Let $N_t$ be the process that starts in $0$ and jumps from $i$ to $i+1$ at rate $f(i)$.
Then, the time of the $k$th jump is in distribution given by 
$S_k =\sum_{i=0}^{k-1} \frac{1}{f(i)} E_i$, where $E_0, E_1, \ldots$ is an i.i.d. sequence of exponentials with rate 1.
Following Cor.~50 in~\cite{Bhamidi2007}), let $Y$ be an independent exponential random variable, so that
\[ \P ( N_{Y} \geq k ) = \P ( Y \geq S_k )  .\] 
By first conditioning on $S_k$, we get
\begin{align} 
 \P ( Y \geq S_k ) & = \E \big[ \E[\mathbbm{1}\{ Y \geq S_k\} \, |\, S_k ] \big]= \E[ e^{-S_k} ] 
= \E[ e^{ - \sum_{i=0}^{k-1} \frac{1}{f(i)} E_i } ] \notag\\
& 
= \prod_{i =0}^{k-1} \E [ e^{- E_i / f(i) } ] = \prod_{i=0}^{k-1} \frac{f(i)}{1+f(i)},\label{eq_mu1}
\end{align}

using that $\E[ e^{-\lambda E_i} ] = \frac{1}{1+\lambda}$ in the last step.
This implies
\begin{align*} \P (N_{Y} = k) & = \prod_{i=0}^{k-1} \frac{f(i)}{1+f(i)} - \prod_{i=0}^{k} \frac{f(i)}{1+f(i)}
\\
&= \left( 1 - \frac{f(k)}{1+f(k)}\right) \prod_{i=0}^{k-1} \frac{f(i)}{1+f(i)} =\frac{1}{1+ f(k)} \prod_{i=0}^{k-1} \frac{f(i)}{1+f(i)} = \mu_k. 
\end{align*}
In particular this shows that $ \mu $ defines a probability measure on $  \N_0 $ and (\ref{eq_mu1}) gives
\begin{equation}
\mu([k,\infty))= \prod_{i=0}^{k-1} \frac{f(i)}{1+f(i)}=f(k-1)\mu_{k-1} \label{eq_mu2}.
\end{equation}
Using this connection to a jump process we obtain by integration by parts
\begin{eqnarray*}
 \sum_{k=1}^\infty g(k) \mu_k &=& \E\bigl[g\bigl(N_{Y}\bigr)\bigr]
                             = \int_0^\infty \E\bigl[g\bigl(N_{s}\bigr)\bigr] e^{-s} \d{s}\\
 &=& g(0) + \int_0^\infty
\E\Bigl[g\Bigl(\bigl(N_{s}+1\bigr)-g\bigl(N_{s}\bigr)\Bigr)f\bigl(N_{s}\bigr)\Bigr]
e^{-s} \d{s}\\
 &=& g(0) +
\E\Bigl[g\Bigl(\bigl(N_{Y}+1\bigr)-g\bigl(N_{Y}\bigr)\Bigr)f\bigl(N_{Y}\bigr)\Bigr].
\end{eqnarray*}
In particular, if $W \sim \mu$, then $W = Z_Y$ and we have just shown that in this case~\eqref{eq:0205-1} holds. 
To show the converse statement suppose $W$ with distribution $(p_k)_{k \in \N_0}$ satisfies~\eqref{eq:0205-1}. Then, by applying~\eqref{eq:0205-1} with $h  = \1_\ell$, we obtain the following recursion
\[ (1+f(0)) p_0 = 1, \quad \mbox{ and } \quad (1+f(\ell) )p_\ell = f(\ell-1) p_{\ell -1}, \mbox{ for } \ell \in \N. \]
Solving the recursion yields   $p_k = \mu_k$ for all $k \in \N_0$.
\end{proof}
Following the generator method introduced by Barbour in \cite{Barbour1988}, the
solution is formally given by
\begin{equation}
g_h(k)= -\int_0^\infty \Bigl( \E_k h(Z_t)-\int h \d\mu\Bigr)  \d t, 
\end{equation}
for all $h$, when the integral exists. Here $Z_t$ denotes the Markov process with infinitesimal generator $\A$,
$\E_k$ its distribution conditional on $Z_0=k$, and $\mu$ its equilibrium distribution.

\begin{lemma}\label{le:stein_sol} The unique solution of the Stein equation for $\mu$, i.e.\
\begin{equation} \mathcal{A} g  = h-\mu(h) ,\label{Steineqaution} \end{equation}
for any $h = \1_A$, with $A \subset \N_0$
 is given 
by 
\begin{equation} g_h(k)= -\int_0^\infty \Bigl( \E_k h(Z_t)-\int h \d\mu\Bigr)  \d t ,\label{Steinsolution}  \end{equation}
where $(Z_t)$ is the continuous-time Markov process with generator $\cA$ started in $k$ under $\p_k$ (resp.\ $\E_k$).
\end{lemma}	
	
	\begin{proof} Let $h = \1_A$ for some $A\subset \N_0$. 
One can check that the Markov chain $(Z_t)$ with generator $\cA$ is	irreducible, non-explosive and has invariant distribution $\mu$. 
It follows that $\E_k h(Z_t) \ra \mu(h)$ as $t \ra \infty$.
Now, applying Proposition 1.5 in \cite[Chapter~1]{EK86} with $ f=h-\mu(h) $, 
gives
\[ h(k) - \E_k [ h(Z_t)] = -  \cA \bigg[ \int_0^t (\E_{\odot} h(Z_s) - \mu(h) ) \, \d s\Big] (k) . \]
Taking $ t \rightarrow \infty $  directly yields that \eqref{Steinsolution} is a solution of \eqref{Steineqaution}. In order to take the limit on the right hand side it remains to show that the integral in \eqref{Steinsolution} exists. We have  since $h = \1_A$,
	\begin{align*}
	\left|g_h(k)\right|&\leq\int_{0}^{\infty}\left|\E_k\left[h(Z_t)\right]-\E\left[h(Z)\right]\right|\d t\\
	&=\int_{0}^{\infty}\left|\E\left[h(Z_t^{(k)})\right]-\E\left[h(Z_t^{(\mu)})\right]\right|\d t\\
	&\leq  \int_{0}^{\infty} \min_{(X_t,Y_t)\text{ coupling of }(Z_t^{(k)}, Z_t^{(\mu)})  }\P(X_t\neq Y_t)\d t.
	\end{align*}
	We will now construct a coupling $ (X_t,Y_t) $ of  $ Z_t^{(k)} $ and $ Z_t^{(\mu)} $ in the following way:
	let $ X_0=k $ and choose $ Y_0 $ according to $ \mu. $ We let both chains evolve independently until $ Y $ falls to zero at a random time $ \tau $. We then force $ X $ to fall to zero as well and from that point on the two chains evolve together, so that $ X_t=Y_t $ for $ t \geq \tau $. One can easily check that this defines a coupling of $ Z_t^{(k)} $ and $ Z_t^{(\mu)}$. As $ Y $ falls down to zero at rate one we get
	\[
	\P(X_t\neq Y_t)\leq \P(\tau\geq t)=e^{-t}
	\] 
	so that
\[		\left|g_h(k)\right|\leq \int_{0}^{\infty}e^{-t}\d t =1,
	\]
	which completes the proof.
\end{proof}

\subsection{Bounding the Stein solution}\label{sectionbounds}

In this section, we exploit the connection to the Markov process with generator $\cA$ to  find bounds on the Stein solution $g_A := g_{\1_A}$ for $A \subset \N_0$. This part of the proof only depends on the limiting distribution $\mu$ and does not make use of the preferential attachment setting.
More precisely, we would like to show that
$v_A(k) := f(k) g_A(k)$ is uniformly bounded in $k$.
As we will see later on, $g_A$ can be written as a linear combination of $g_j := g_{\1_{\{j\}}}$. Therefore, we
start by calculating the latter.

\begin{lemma}\label{Lemma1}
	For $ g_j:=g_{h} $ for $ h=\mathbbm{1}(\cdot =j) $ as in (\ref{Steinsolution}) we have
	\begin{align}
	\Delta g_j(k)=g_j(k+1)-g_j(k)=
	\begin{cases} 
	-\frac{\mu_j}{\mu_{k}} \frac{1}{f(k)(1+f(k))}, &\text{for } j\geq k+1, \\ 
	\frac{1}{1+f(j)}, &\text{for } j=k,\\
	0, &\text{for } j \leq k-1. \label{Steinincrementequal}
	\end{cases}
	\end{align}
		\end{lemma}

	\begin{proof}
	We apply the techniques used in \cite{Brown2001} and adapt them to our Markov process $ Z_t $.
Define
		$$ \tau_{k,k+1}=\inf\{t: Z_t^{(k)}=k+1\}, $$ 
		where $  Z_t^{(k)} $  denotes a Markov process with generator $ \mathcal{A} $ started in $ k $. 
Then,  for $ k \leq j-1 $  we obtain via the representation~\eqref{Steinsolution} of the Stein solution,
				\begin{align*}
		g_j(k)&=- \int_{0}^{\infty}\big(\E \big[\1_{\{j\}}(Z_t^{(k)})\big] - \mu_{j}\big)dt \\
		&= -\E \Big[ \int_{0}^{\tau_{k,k+1}}\Big(\1_{\{j\}}(Z_t^{(k)})- \mu_{j}\Big)dt\Big] - \E \Big[\int_{\tau_{k,k+1}}^{\infty}\Big(\1_{\{j\}}(Z_t^{(k)}) - \mu_{j}\Big)dt \Big]\\
		&=\mu_{j}\E\left[\tau_{k,k+1}\right]+ g_j(k+1),
		\end{align*}
		where the last  equality uses the strong Markov property of $ Z_t $. 
Rearranging yields
		\begin{eqnarray*}
			g_j(k)-g_j(k+1)=\mu_{j}\E\left[\tau_{k,k+1}\right] \geq 0 
		\end{eqnarray*}
		for $ k \leq j-1 $. 
		Following the same procedure 
		for $ k \geq j+1 $ and $ \tau_{k,0}=\inf\{t: Z_t^{(k)}=0\} $ we get
		\begin{align*}
		g_j(k)&=- \int_{0}^{\infty}\big(\E \big[\1_{\{j\}}(Z_t^{(k)})\big] - \mu_{j}\big) \, dt \\
		&= -\E \Big[ \int_{0}^{\tau_{k,0}}
(\1_{\{j\}}(Z_t^{(k)})- \mu_{j})\, dt\Big] - \E \Big[\int_{\tau_{k,0}}^{\infty}(\1_{\{j\}}(Z_t^{(k)}) - \mu_{j})\, dt \Big]
\\
		&=\mu_{j}\E\left[\tau_{k,0}\right]+ g_{j}(0)
		\end{align*}
		and thus
		\begin{eqnarray*}
			g_j(k)-g_{j}(0)=\mu_{j}\E\left[\tau_{k,0}\right]
		\end{eqnarray*}
		for $ k \geq j+1 $.
		Noticing that  $ \E\left[\tau_{k,0}\right]=1 $ for all $ k $, since the rate by which the process $ Z_{t} $ falls down to 0 is 1, independent of the current  state of the process, this simplifies to
		\begin{eqnarray*}
			g_j(k)-g_{j}(0)=\mu_{j}.
		\end{eqnarray*}
		This means that $ g_j(k) $ is constant for $ k\geq j+1 $, so that
		$$
		g_j(k+1)-g_j(k)=0, \text{ for } k\geq j+1.
		$$
		Furthermore, we have
		\begin{align*}
		g_j(j+1)&=- \int_{0}^{\infty}\big(\E \big[\1_{\{j\}}(Z_t^{(j+1)})\big] - \mu_{j}\big) \, dt\\
		&=-\E \Big[ \int_{0}^{\tau_{j+1,j}}(\1_{\{j\}}(Z_t^{(j+1)}) - \mu_{j} ) \, dt\Big] - \E \Big[\int_{\tau_{j+1,j}}^{\infty}(\1_{\{j\}}(Z_t^{(j+1)}) - \mu_{j}) \, dt \Big]
\\
		&=\mu_{j}\E\left[\tau_{j+1,j}\right]+ g_j(j)\\
		&=\mu_{j}\big( \E[\tau_{j+1,0}]+\E[\tau_{0,j}]\big)+ g_j(j)\\
		&=\mu_{j}\left(1+\E\left[\tau_{0,j}\right]\right)+ g_j(j),
		\end{align*}
		yielding
		\begin{eqnarray*}
			g_j(j+1)-g_j(j)=\mu_{j}\left(1+\E\left[\tau_{0,j}\right]\right)= \mu_j \E\left[\tau_{j,j}\right] \geq 0, 
		\end{eqnarray*}
where $\tau_{j,j}$ defines the first return time to $j$ of the Markov chain started at $j$.
		
	Thus the following equations hold for the Stein solution $ g_j $:
		\begin{align}
		g_j(k+1)-g_j(k)=
		\begin{cases} 
		-\mu_j \E\left[\tau_{k,k+1}\right], &\text{for } j\geq k+1, \\ 
		\mu_j \E\left[\tau_{j,j}\right], &\text{for } j=k,\\
		0, &\text{for } j \leq k-1. 
		\end{cases}\label{Steinincrements}
		\end{align}

We can simplify this expression further as follows. Let $S_k$ be the first jump time of $Z_t^{(k)}$, then 
by definition of the Markov chain, we have that $ S_k \sim {\rm Exp}(1+f(k)) $. 
Since $Z_t^{(k)}$ jumps to $k+1$ at rate $f(k)$ and to $0$ at rate $1$ we obtain
		\begin{align*}
		\E\left[\tau_{k,k+1}\right] 
& = \E[S_k] + \frac{1}{1+f(k)} \E[\tau_{0,k+1}] = \frac{1}{1+f(k)}( 1 + \E[\tau_{0,k}] + \E[\tau_{k,k+1}] ) 
		\end{align*}
		and thus
		\begin{align}
\E\left[\tau_{k,k+1}\right] &=\frac{1}{f(k)}\big(1+\E\left[\tau_{0,k}\right]\big)=\frac{1}{f(k)}\E\left[\tau_{k,k}\right].\label{stepk}
		\end{align}
A standard result for Markov chains, see e.g.~\cite[Thm.\ 3.6.3]{Norris}, yields that
		\begin{equation}
		\mu_{j}= \frac{1}{(1+f(j))\E \left[\tau_{j,j}\right]}\label{returntime}
		\end{equation}
Rearranging~\eqref{returntime} and 
combining it with (\ref{Steinincrements}) and (\ref{stepk}) yields the statement  of the lemma.
	\end{proof}

\begin{prop}\label{prop:bound_v_A} For any $k \in \N_0$ and $A \subset \N_0$, we have
\[ |v_A(k)| \leq 1 . \]
\end{prop}

\begin{proof}
First note that
	\begin{align}
	g_A(k)&=- \int_{0}^{\infty} \big( \E\big[\1(Z_{t}^{(k)}\in A)\big]-\mu(A)\big) \, dt \notag\\
	&=- \int_{0}^{\infty} \Big(\E\Big[\sum_{j \in A}\1(Z_{t}^{(k)}=j)\Big]-\sum_{j \in A}\mu_j\Big) \, dt\notag\\
	&=\sum_{j \in A}\left(- \int_{0}^{\infty} \big( \E\big[\1(Z_{t}^{(k)}=j)\big]-\mu_j \big) \, dt\right)
	=\sum_{j \in A}g_j(k) .  \label{sumSteinsolution}
	\end{align}
Consequently, we have by Lemma~\ref{Lemma1}
\[ v_A(k) = \sum_{j \in A} f(k) \Delta g_j(k) 
= - \frac{f(k)}{\mu_k f(k) (1+f(k))} \mu( A \cap [k,\infty)) + \frac{f(k)}{1+f(k)}\1_{k \in A} . \]
Using that by \eqref{eq_mu2} $\mu([k,\infty)) = \prod_{i=0}^{k-1}\frac{f(i)}{1+f(i)}=(1+f(k)) \mu_k$ , we obtain 
\[v_A(k) =   - \frac{1}{\mu([k,\infty))} \mu( A \cap [k,\infty)) + \frac{f(k)}{1+f(k)}\1_{k \in A} ,  \]
so that the proposition follows immediately.
\end{proof}

\subsection{Applying Stein's method to the preferential attachment model}

We use the following Markov chain to describe the evolution of the indegree of a uniform vertex. Similar ideas were also used in~\cite{Dereich2009, F09}.

\begin{lemma}\label{LemmaMC}
	Let $ X_n $ be a Markov chain with $ \P(X_1=0)=1 $ and  transition probabilities
given for any $i \geq 1$ as
\[ \p(X_{n+1} = j \, | \, X_n = i) = \left\{ \begin{array}{ll} \frac{f(i)}{n+1} & \mbox{ if } j = i+1,  \\[2mm] 
\frac{n- f(i)}{n+1} & \mbox{ if } j = i, \\[2mm] \frac{1}{n+1} & \mbox{ if } j = 0, \end{array} \right. \]
and
\[ \p(X_{n+1} = j \, | \, X_n = 0) = \left\{ \begin{array}{ll} \frac{f(0)}{n+1} & \mbox{ if } j = 1, \\[2mm] 
1 - \frac{f(0)}{n+1} & \mbox{ if } j = 0 .\end{array} \right. \]
	Then $ \mathcal{L}(X_n)=\mathcal{L}(W_n) $, where $ W_n $ denotes the indegree of the uniformly chosen vertex in any preferential attachment model at time $ n $ satisfying Assumptions (A) with $d_0 = 0$.
\end{lemma}
Note that the Markov chain starts at time $1$ to match the index of the random graph.

\begin{proof}
Consider the Markov chain $ (J_{n})_{n \in \N} $ started with $J_1 =1$ and such that for $n\geq1$
\begin{align*}
 \P(J_{n+1}=J_{n}|J_{n})=\frac{n}{n+1} \quad \mbox{and}\quad
 \P(J_{n+1}=n+1|J_{n})=\frac{1}{n+1}.
\end{align*}
Then it is straight-forward to check by induction that 
 $ J_{n} $ is uniformly distributed on $  \{1, \ldots, n\} $. 
We now set $ X_n:= \deg^-_n(J_n) $, so that in particular $ \mathcal{L}(W_n)=\mathcal{L}(X_n) $. 
  Then using the dynamics of the preferential attachment model and the tower property, the following transition probabilities hold for $ X_n $ and $ 1 \leq j \leq n $
\begin{align*}
\P  (X_{n+1}=j+1|X_{n}=j)&=\frac{f(j)}{n}\times \frac{n}{n+1}=\frac{f(j)}{n+1},\\
\P  (X_{n+1}=j|X_{n}=j)&=(1-\frac{f(j)}{n})\times \frac{n}{n+1}=\frac{n-f(j)}{n+1},\\
\P  (X_{n+1}=0|X_{n}=j)&=\frac{1}{n+1}. 
\end{align*}
Moreover, for $ j=0 $,
\begin{align*}
\P  (X_{n+1}=1|X_{n}=0)&=\frac{f(0)}{n}\times \frac{n}{n+1}=\frac{f(0)}{n+1},\\
\P  (X_{n+1}=0|X_{n}=0)&=(1-\frac{f(0)}{n})\times \frac{n}{n+1}+\frac{1}{n+1}=1-\frac{f(0)}{n+1}.
\end{align*}
This completes the proof of the lemma.
\end{proof}

Following the discussion above in Section \ref{sec:proof_strategy}, the next step is to find a bound on 
\[\begin{aligned}  \E\left[\mathcal{A}g_{A}(W_{n+1})\right] &=\E\left[\mathcal{A}g_{A}(X_{n+1})\right]\\
& =\E\left[f(X_{n+1})\Delta g_{A}(X_{n+1})+ \left(g_{A}(0)-g_{A}(X_{n+1})\right)\right], \end{aligned} \]
where we recall that $ \Delta g_{A}(k):=g_{A}(k+1)-g_{A}(k) $.

\begin{lemma}\label{le:recursion}
	For the Markov chain $(X_n)_{n \geq 0}$ defined in Lemma \ref{LemmaMC} and $ v_A(k):=f(k)\Delta g_A(k) $, with $ A \subset \N_0 $, we have
	\begin{align}
	&\E\left[\mathcal{A}g_{A}(X_{n+1})\right]\notag\\
	&=\frac{1}{n+1}\Big(\Big( \sum_{\ell=1}^{n} \sum_{k = 0}^{\ell -1}
	\Delta v_A(k) \big( f(k)\P(X_{\ell}= k)- \P(X_{\ell}\geq k+1)\big) \Big)+v_A(0)\Big). \label{Triplesum}
	\end{align}
	\end{lemma}  

\begin{proof}
Let $h: \N_0 \rightarrow \R$ be such that $h(0) = 0$. Then, 
\begin{align*}
	\E\left[h(X_{n+1})\right] & = \E\left[\E\left[h(X_{n+1}) \, |\,X_{n}\right]\right]                             \\
	                               & =\E\Big[h(X_{n})\, \frac{n-f(X_{n})}{n+1}+h(X_n + 1) \, \frac{f(X_{n})}{n+1}\Big] \\
& = \frac{n}{n+1} \E[ h(X_n)] + \frac{1}{n+1} \E[ f(X_n) \Delta h(X_n)  ] .
\end{align*}
By iteration and using that $X_1 = 0$, we get
\begin{align*}
\E\left[h(X_{n+1})\right]&=\frac{1}{n+1}\sum_{\ell=1}^{n} \E[f(X_{\ell})\Delta h(X_{\ell})].
\end{align*}
Define\[ h(k) = \mathcal{A} g_A(k) - v_A(0)= v_A(k)  + (g_A(0) - g_A(k) ) - v_A(0). \]
Thus, we obtain 
\begin{align}  \E[ \mathcal{A} g_A(X_{n+1}) ] & = \E [ h(X_{n+1}) ] + v_A(0) \notag\\
& = \frac{1}{n+1} \sum_{\ell = 1}^n \Big( \E[ f(X_\ell) \Delta v_A(X_\ell) ]  - \E[ f(X_\ell) \Delta g_A(X_\ell)]\Big) +v_A(0)  \notag\\
& = \frac{1}{n+1} \sum_{\ell = 1}^n \sum_{k = 0 }^{\ell-1} 
f(k) \Delta v_A(k) \p( X_\ell  = k) \notag\\
& \qquad- \frac{1}{n+1} \sum_{\ell = 1}^n 
\sum_{k=0}^{\ell -1} (v_A(k) - v_A(0) )  \p(X_\ell =k)  + \frac{v_A(0) }{n+1}.
\label{eq:2202-1}
\end{align}
For the second sum, we write
\[ \begin{aligned}  \sum_{k=0}^{\ell -1} & (v_A(k) - v_A(0))  \p (X_\ell = k) = 
\sum_{k=0}^{\ell -1} \sum_{i =0}^{k-1} \Delta v_A(i) \p (X_\ell = k) \\
&  =  \sum_{i=0}^{\ell -1} \Delta v_A(i) \sum_{k=i+1}^{\ell -1} \p (X_\ell = k) 
 =  \sum_{i=0}^{\ell -1} \Delta v_A(i) \p (X_\ell \geq i+1).
\end{aligned}
\]
Combining the latter with~\eqref{eq:2202-1} yields the statement of the lemma.
\end{proof}

\subsection{Bounding the difference $ f(k)\P(X_l=k)-\P(X_l\geq k+1)$}

\begin{prop}\label{prop_difference}
Suppose $f$ satisfies $f(k) \leq k+1$. Define
\[ h(k,\ell) := f(k)\P(X_\ell=k)-\P(X_\ell\geq k+1), \]
where $(X_\ell)_{\ell \geq 1}$ is the Markov chain from Lemma~\ref{LemmaMC}.
\begin{itemize}
\item[(i)] Then, for any $k \in \N_0,\ell \in \N$ we have
\[ h(k, \ell) \geq 0, \]
and for $k \geq \ell$, we have $h(k,\ell) = 0$.
	\item[(ii)] Suppose there exists $K$ such that  $k \leq f(k) \leq k+1$ for all $0 \leq k \leq K$, then  for all $\ell \leq K+1$,
	we have
	\begin{equation}\label{eq:2302-1}
		h(k+1,\ell)-h(k,\ell)\geq 0 \quad \mbox{for all } k \leq \ell -2 
	\end{equation}
	\item[(iii)] Assume that there exist $k_\ast$ such that  $f(k) \leq k$ for all $k \geq k_\ast$ and $f(k)> k$ for $k<k_\ast$. Then,  for all $\ell \in \N$ there exists $I(\ell) \in \{ 0, \ldots, \ell -1\}$ 
	such that
	\begin{equation}\label{eq:2302-2} h(k+1,\ell) - h(k,\ell)  \left\{ \begin{array}{ll} \geq 0 & \mbox{if } k < I(\ell)  , \\ \leq 0 & \mbox{if } k \geq I(\ell). \end{array} \right.  \end{equation}
	Moreover, $I(\ell+1) \in \{ I(\ell), I(\ell) + 1\}$.
	\item[(iv)] Assume there exists $k_\ast \in \N_0$ such that $f(k) \leq k$ for all $k\geq k_\ast$. Then, 
there exists a constant $C > 0$ such that for all $k \in \N_0, \ell \in \N$,
	\[ h(k,\ell) \leq \frac{C}{\ell} .\]
\item[(v)] If there exists $\gamma \in (0,1)$ such that $f(k) \in [k,k+\gamma]$ for all $k \in \N_0$, then 
\[ \sup_{k \in \N_0} h(k,\ell) \leq C \ell^{-(1-\gamma)} . \]
\end{itemize}
\end{prop}

Before we start with the proof of the proposition, we derive a recursion for the coefficients $h$ and also for its increments.

\begin{lemma}\label{le:recursion2} Let $h$ be defined as in Proposition~\ref{prop_difference}, then $h(k,\ell)  = 0$ for all $k \geq \ell$ and for 
 all $\ell \in \N$, $k \in \N_0$, we have
\begin{equation}\label{eq:recursion}
h(k,\ell+1)  = \left(\frac{\ell}{\ell+1}-\frac{f(k)}{\ell+1}\right)h(k,\ell)+\frac{f(k)}{\ell+1}h(k-1,\ell),
\end{equation}
where we define $h(-1,\ell) = 0$.
Moreover, if we define $\Delta^\ssup{1} h(k,\ell) := h(k+1,\ell) - h(k,\ell)$, we have that for all $\ell \in \N$ and $k \leq \ell -1$
\begin{equation}\label{eq:deriv} \Delta^\ssup{1} h(k,\ell + 1) =  
\left(\frac{\ell}{\ell+1}-\frac{f(k+1)}{\ell+1}\right)\Delta^\ssup{1} h(k,\ell)+\frac{f(k)}{\ell+1}\Delta^\ssup{1} h(k-1,\ell)
\end{equation}
\end{lemma}

\begin{proof}  Note that since $X_\ell \leq \ell -1$ \ $\P$-a.s, we have $h(k,\ell) = 0$ 
for any $k \geq \ell$. 
Moreover, 
 by the definition of the Markov chain $(X_n)$, for $k\geq 1$,
\begin{align*}
h(k,\ell+1)
&=f(k)\P(X_{\ell+1}=k)-\P(X_{\ell+1}\geq k+1)\notag\\
&=f(k)\Big(\Big(1-\frac{f(k)+1}{\ell+1}\Big)\P(X_\ell=k)+\frac{f(k-1)}{\ell+1}\P(X_\ell=k-1)\Big)\notag\\
&\hspace{0.5cm}{}-\Big(\frac{\ell}{\ell+1}\P(X_\ell\geq k+1)
+\frac{f(k)}{\ell+1}\P(X_\ell=k)\Big)\notag\\
& = \left(\frac{\ell}{\ell+1}-\frac{f(k)}{\ell+1}\right)h(k,\ell) - \frac{f(k)}{\ell+1} \p (X_\ell \geq k+1) \\
& \hspace{0.5cm}+ \frac{f(k)}{\ell+1} \big( f(k-1) \p(X_\ell = k-1) - \p(X_\ell = k) \big) \\
&=\left(\frac{\ell}{\ell+1}-\frac{f(k)}{\ell+1}\right)h(k,\ell)+\frac{f(k)}{\ell+1}h(k-1,\ell).
\end{align*}
Note for $k = 0$, we have
\[\begin{aligned}  h(0,\ell +1) & = f(0) \p( X_{\ell + 1} = 0 ) - \p (X_{\ell+1} \geq 1) \\
& = f(0) \Big( \Big( 1 - \frac{f(0)}{\ell+1} \Big) \p (X_\ell = 0 ) 
+ \frac{1}{\ell+1} \p (X_{\ell} \geq 1) \Big) \\
& \hspace{1cm}- \frac{f(0)}{\ell+1} \p (X_{\ell} = 0 ) - \Big( 1 - \frac{1}{\ell+1}\Big) \p (X_\ell \geq 1)  \\
& = \Big( 1 - \frac{f(0)+1}{\ell+1} \Big) h(0,\ell) 
=  \Big( \frac{\ell}{\ell+1} - \frac{f(0)}{\ell+1} \Big) h(0,\ell) .
\end{aligned} \]
Therefore, the identity~\eqref{eq:recursion} also holds for $k = 0$ since
we defined $h(-1,\ell) = 0$ for all $\ell \in \N$.

Note that by~\eqref{eq:recursion},
\begin{align*}
&h(k+1,l+1)-h(k,l+1)\\&=\left(\frac{\ell}{\ell+1}-\frac{f(k+1)}{\ell+1}\right)h(k+1,\ell)+\frac{f(k+1)}{\ell+1}h(k,\ell)\\ &\qquad   -\left(\frac{\ell}{\ell+1}-\frac{f(k)}{\ell+1}\right)h(k,\ell)-\frac{f(k)}{\ell+1}h(k-1,\ell)\\
&=\left(\frac{\ell}{\ell+1}-\frac{f(k+1)}{\ell+1}\right)\left(h(k+1,\ell)-h(k,\ell)\right)+\frac{f(k)}{\ell+1}\left(h(k,\ell)-h(k-1,\ell)\right), 
\end{align*}
which shows~\eqref{eq:deriv}.
\end{proof}

\begin{proof}[Proof of Proposition~\ref{prop_difference}]
Before we start with the proof, note that for $\ell = 2$, we have
\begin{equation}\begin{aligned}\label{eq:h02}
  h(0,2)&=f(0)\P(X_2=0)-\P(X_2\geq 1)\\
  & =f(0)\bigg(1-\frac{f(0)}{2}\bigg)-\frac{f(0)}{2}
  =\frac{f(0)(1- f(0))}{2},
\end{aligned}\end{equation}
and moreover, 
\begin{equation}\label{eq:h12} h(1,2) = f(1) \p(X_2 = 1) = \frac{f(1) f(0)}{2} . \end{equation}

(i) We now show that for any $\ell \in \N$:
$ h(k,\ell) \geq 0$ for any $k \in \N_0$ by induction on~$\ell$.
The induction hypothesis follows from~\eqref{eq:h02} and \eqref{eq:h12} since $f(0) \leq1$.
We now assume that the statement holds for some $\ell$, then from~\eqref{eq:recursion} and
for $k \leq \ell -1$ we have
\[ h(k, \ell+1) = \Big(\frac{\ell}{\ell + 1} - \frac{f(k)}{\ell +1} \Big) h(k, \ell)
+ \frac{f(k)}{\ell +1} h(k-1, \ell) , \]
which is nonnegative  due to the induction hypothesis and the condition that $ f(k)\leq k+1 \leq \ell $. For $ k=\ell $ we have
$h(\ell,\ell+1) = f(\ell) \p(X_{\ell+1}) \geq 0$. This implies the induction step since all other terms are $0$.

(ii) Let now $ f $ be such that there exists $K$ such that $ k\leq f(k)\leq k+1 $ for all $k \leq K$. As before we will use induction on $\ell$ to show the stated result. For $ \ell =2 $ we get from~\eqref{eq:h02} and~\eqref{eq:h12}
$$
h(1,2)-h(0,2)=\frac{f(0)}{2} ( f(0) + f(1) - 1) \geq 0,
$$
as $ f(1)\geq 1 $ by assumption.

Suppose that the statement~\eqref{eq:2302-1} is true for some $ \ell  \leq K $. By~\eqref{eq:deriv} we obtain
\[ \Delta^\ssup{1} h(k,\ell + 1) =  
\left(\frac{\ell}{\ell+1}-\frac{f(k+1)}{\ell+1}\right)\Delta^\ssup{1} h(k,\ell)+\frac{f(k)}{\ell+1}\Delta^\ssup{1} h(k-1,\ell),\]
which is nonnegative for $k \leq \ell -2$ due to the induction hypothesis and since $ f(k+1)\leq f(\ell-1)\leq \ell $. It remains to show that $\Delta^\ssup{1} h(\ell-1,\ell+1) \geq 0$. Again by~\eqref{eq:deriv} and using that $\Delta^\ssup{1} h(\ell-1,\ell) = - h(\ell-1,\ell)$, we get that
\[ \begin{aligned} \Delta^\ssup{1} h(\ell -1,\ell + 1) & =  
\left(\frac{\ell}{\ell+1}-\frac{f(\ell)}{\ell+1}\right)(- h(\ell-1,\ell))+\frac{f(\ell -1)}{\ell+1}\Delta^\ssup{1} h(\ell-2,\ell)\\
&  = 
\frac{f(\ell) - \ell}{\ell +1} h(\ell-1,\ell) + \frac{f(\ell -1)}{\ell+1}\Delta^\ssup{1} h(\ell-2,\ell) , \end{aligned} 
\]
which is nonnegative by induction hypothesis and since $f(\ell) \geq\ell$.

(iii) Again we show by induction on $\ell$ that there exists $ I(\ell) \in \{ 0,\ldots, \ell-1\} $ such that~\eqref{eq:2302-2} is valid. Note for $\ell =2$ the statement holds trivially. Moreover,  for $ \ell \leq k_\ast $ 
 the statement holds by (ii) with $I(\ell) = \ell -1$.

Suppose the statement~\eqref{eq:2302-2} is true for some $\ell\geq k_\ast $. 
By~\eqref{eq:deriv} we obtain
\[ \Delta^\ssup{1} h(k,\ell + 1) =  
\left(\frac{\ell}{\ell+1}-\frac{f(k+1)}{\ell+1}\right)\Delta^\ssup{1} h(k,\ell)+\frac{f(k)}{\ell+1}\Delta^\ssup{1} h(k-1,\ell).\]
From this we can deduce that if $k < I(\ell)$, then since $I(\ell) \leq \ell -1$, 
we have $f(k+1) \leq k+2 \leq \ell$, so that $\Delta^\ssup{1} h(k,\ell+1) \geq 0$. 
Conversely, if $k > I(\ell)$ and $k \leq \ell -2$, we can deduce similarly that
$\Delta^\ssup{1} h(k,\ell +1) \leq 0$. It remains to show that $\Delta h^\ssup{1}(k,\ell+1) \leq 0$ for $k = \ell, \ell -1$ if $\ell-1 > I(\ell)$. Note the
case $k = \ell$ holds since $\Delta^\ssup{1} h(\ell,\ell+1) = - h(\ell, \ell+1)$.
By the recursion~\eqref{eq:recursion} together with~\eqref{eq:h12}, we get 
\begin{equation}\label{eq:right} h(\ell,\ell+1) = \frac{f(\ell)}{\ell+1} h(\ell-1,\ell)  = f(\ell) \prod_{i=1}^\ell \frac{f(i-1)}{i+1} .\end{equation}
Now, since $\Delta^\ssup{1} h(\ell-2,\ell) \leq 0$, we have again by~\eqref{eq:recursion},
\[\begin{aligned}
\Delta^\ssup{1} & h(\ell-1,\ell+1)\\
 &  = h(\ell, \ell+1) - \Big( \frac{\ell}{\ell+1} - \frac{f(\ell-1)}{\ell +1} \Big) h(\ell-1,\ell)  - \frac{f(\ell-1)}{\ell+1} h(\ell -2, \ell) \\
& \leq h(\ell, \ell+1) - \frac{\ell}{\ell+1} h(\ell-1,\ell) \\
& = f(\ell) \prod_{j=1}^{\ell} \frac{f(j-1)}{j+1} - \frac{\ell}{\ell+1} 
f(\ell -1) \prod_{j=1}^{\ell-1} \frac{f(j-1)}{j+1} \\
& = \frac{f(\ell - 1)}{\ell+1} ( f(\ell ) - \ell) \prod_{j=1}^{\ell-1} \frac{f(j-1)}{j+1}   ,  \end{aligned}  \]
which is negative as $f(\ell) \leq \ell$ since $ \ell\geq k_\ast$.
In particular, we have seen that $I(\ell+1) \in \{ I(\ell), I(\ell) + 1\}$.

(iv) Define
\[ C := f(0) \Big( 1 \vee \max_{1 \leq k \leq k^\ast} \prod_{i=1}^k \frac{f(i)}{i} \Big). \]
Then, we will show inductively that for all $\ell \in \N$:
\begin{equation}\label{eq:2003-1}
h(k,\ell)\leq \frac{C}{\ell},\quad \mbox{for all } k \leq \ell -1.
\end{equation}
For $\ell =1$, we have
\[ h(0,1) = f(0 ) \p (X_1 = 0) = f(0) \leq C . \]

Now, assume that~\eqref{eq:2003-1} holds for some $\ell \in \N$. 
 Using the identity \eqref{eq:recursion} we obtain for $ k \leq \ell-1 $
\begin{align*}
h(k,\ell+1) 
&=\left(\frac{\ell}{\ell+1}-\frac{f(k)}{\ell+1}\right)h(k,\ell)+\frac{f(k)}{\ell+1}h(k-1,\ell)\\
&\leq \left(\frac{\ell}{\ell+1}-\frac{f(k)}{\ell+1}\right)\frac{C}{\ell} +\frac{f(k)}{\ell+1}\frac{C}{\ell} 
=\frac{C}{\ell+1},
\end{align*}
where we used that $f(k) \leq k+1 \leq \ell$ in the second step. 
For $k =\ell$, we have by~\eqref{eq:right}
\[ h(\ell, \ell+1) = f(\ell) \prod_{i=1}^\ell \frac{f(i-1)}{i+1} = \frac{f(0)}{\ell+1} \prod_{i=1}^\ell \frac{f(i)}{i}  . \]
Then, if $\ell \leq k^*$, this is trivially bounded by $C/(\ell +1)$. Furthermore, if $\ell > k^\ast$, then 
\[ h(\ell, \ell+1) = \frac{f(0)}{\ell+1} \prod_{i=1}^\ell \frac{f(i)}{i} \leq \frac{C}{\ell+1} \prod_{i = k^\ast+1}^k \frac{f(i)}{i} \leq \frac{C}{\ell+1}, \]
since $f(i) \leq i$ for all $i \geq k^\ast$. This completes the induction step.

(v) Note that by (ii), $k \mapsto h(k,\ell)$ is increasing for $k \leq \ell -1$. In particular, we have by (i)
\[ \sup_{k \in \N_0} h(k,\ell) = \sup_{k \leq \ell -1} h(k,\ell) = h(\ell-1,\ell). \]
By~\eqref{eq:right}, we get that
\[ h(\ell-1,\ell) = f(\ell-1) \prod_{i=1}^{\ell -1} \frac{f(i-1) }{i+1} 
\leq \frac{\prod_{i=0}^{\ell-1} (i+\gamma)}{\ell!} = \frac{1}{\Gamma(\gamma)}\frac{\Gamma(\ell +\gamma)}{\Gamma(\ell+1)}
\sim \frac{1}{\Gamma(\gamma)} \ell^{\gamma -1}, 
\]
using the asymptotics for the Gamma function. This immediately gives statement (v).
\end{proof}

\subsection{Proofs of Theorems~\ref{mainthm2} and~\ref{mainthm}}

We can combine our previous estimates to prove the two main theorems simultaneously.

\begin{proof}[Proofs of Theorems~\ref{mainthm2} and~\ref{mainthm}]
We first consider the case that the preferential attachment model satisfies Assumption (A) with $d_0 = 0$ so that we can generate the indegree of a uniform vertex using the Markov chain $(X_{\ell})_{\ell \geq 1}$ defined in Lemma~\ref{LemmaMC}.
Using the notation $h(k,\ell) = f(k) \p (X_\ell = k) - \p(X_\ell \geq k+1)$, we have from 
Lemma~\ref{le:recursion}, for any $A \subset \N_0$,
\[ \E[ \mathcal{A} g_A(X_{n+1}) ] 
= \frac{1}{n+1}\Big(\Big( \sum_{\ell=1}^{n} \sum_{k = 0}^{\ell -1}
	\Delta v_A(k) h(k,\ell) \Big)+v_A(0)\Big). \]
Using a discrete integration by parts formula and the fact that $h(\ell,\ell) = 0$, we can rewrite the inner sum as
\[  \sum_{k = 0}^{\ell -1}
	\Delta v_A(k) h(k,\ell) =  - v_A(0) h(0, \ell) - \sum_{k=0}^{\ell-1} v_A(k+1) \Delta^\ssup{1} h(k,\ell) . 
\]
Under the assumptions of Theorem~\ref{mainthm2} we can apply Proposition~\ref{prop_difference}(iii) and for Theorem~\ref{mainthm} part (ii) of Proposition~\ref{prop_difference} to deduce that in both cases there exist
$I(\ell)$ such that $\Delta^\ssup{1} h(k,\ell) \geq 0$ for $k < I(\ell)$ and 
$\Delta^\ssup{1} h(k,\ell) \leq 0$ for $k \geq I(\ell)$. 
In particular, we have that 
\[\begin{aligned}  \Big|\sum_{k=0}^{\ell-1}  & v_A(k+1) \Delta^\ssup{1} h(k,\ell) \Big| \\
& \leq \sup_{k} |v_A(k)| \Big(h(I(\ell), \ell) - h(0,\ell) + h(I(\ell), \ell) - h(\ell,\ell)\Big) \\
& \leq 2 \sup_{k} |v_A(k)|  \sup_{k \leq \ell -1} h(k,\ell) . \end{aligned}\]
Therefore, we get
 \[ \begin{aligned} \big| \E[ \mathcal{A}  & g_A(X_{n+1}) ] \big|\\
& \leq \frac{|v_A(0)|}{n+1}+\Big| \frac{1}{n+1} \sum_{\ell=1}^{n}\Big(  \Big( \sum_{k = 0}^{\ell -1}
	 v_A(k+1) \Delta^\ssup{1} h(k,\ell) \Big)  +v_A(0) h(0,\ell)  \Big) \Big| \\
& \leq  \frac{|v_A(0)|}{n+1}+\frac{2}{n+1}  \sup_{k} |v_A(k)| \ \sum_{\ell=1}^{n} \sup_{k \leq \ell-1} h(k,\ell)  .
\end{aligned}\]
Hence, if we combine this estimate with Proposition~\ref{prop:bound_v_A} we obtain that 
\[ d_{\rm TV}(W_{n+1}, W) = \sup_{A \subset \N_0} 
\big| \E[ \mathcal{A} g_A(X_{n+1}) ] \big|
\leq \frac{1}{n+1}+\frac{2}{n+1}   \ \sum_{\ell=1}^{n} \sup_{k \leq \ell-1} h(k,\ell). 
\]
Finally, we note that in the case of Theorem~\ref{mainthm2}, we can apply Proposition~\ref{prop_difference} (iv) to deduce that in this case, there exists a constant $C>0$ such that 
\[ d_{\rm TV}(W_{n+1}, W)  \leq \frac{1}{n+1}+\frac{2\, C}{n+1}      \sum_{\ell=1}^{n} \frac{1}{\ell} , \]
which immediately gives the required bound. In the case of Theorem~\ref{mainthm}, we can 
instead apply Proposition~\ref{prop_difference} (v) to get a constant $C > 0$ such that 
\[ d_{\rm TV}(W_{n+1}, W)  \leq \frac{1}{n+1}+\frac{2\, C}{n+1}    \sum_{\ell=1}^{n} \ell^{- (1-\gamma)} , \]
which again yields the statement of the theorem.

Finally, we consider the case when the model satisfies Assumptions (A) with 
$d_0 > 0$. In this case, by the same argument as in Lemma~\ref{LemmaMC}, 
the indegree of a uniformly chosen vertex has the same distribution as a Markov chain $(\tilde X_n)_{n \geq 1}$ with $\tilde X_1 = d_0$, but the same transition probabilities as $(X_n)_{n \geq 1}$. Let $\tau = \inf\{ k \geq 2\, : \, \tilde X_k = 0\}$. We can couple $(X_n), (\tilde X_n)$ by first letting $\tilde X_n$ evolve and then letting $X_n$ evolve independently until time $\tau$.  Further, we  set $X_k := \tilde X_k$ for all $k \geq \tau$.
By the characterization of $d_{\rm TV}$ in terms of couplings, we thus have
\[ d_{\rm TV} (X_n , \tilde X_n) \leq \p (\tau > n)  
= \prod_{i=1}^n  \Big( 1- \frac{1}{i+1}\Big)=\frac{1}{n+1}. \]
By the first part of the proof, this completes the proof also for $d_0 > 0$.
\end{proof}

\section{Rate of convergence for the outdegree}\label{outdegree}

We now consider the model introduced in Example~\ref{ex:1}, where connections to old vertices are made independently, and 
prove Theorem~\ref{mainthm3}, i.e.\ show the rate of convergence  for the outdegree $D_n = \deg^+_n(n)$ of vertex $n$.

We will need the following moment bound.

\begin{lemma}\label{le:moment}
For the preferential attachment model as in Example~\ref{ex:1} with $f(k) \leq \gamma k + 1$ for all $k$, and some $\gamma \in (0,1)$, we have for all $n \in \N$,
\[ \E[ f(\deg^-_n(i))] \leq  \Big( \frac{n}{i} \Big)^\gamma, \quad \mbox{for all } i \in [n] . \]
\end{lemma}

\begin{proof} For the linear attachment rule $f^\ssup{\gamma} := \gamma \cdot + 1$, the statement of the lemma is proved in~\cite[Lemma 2.7]{Dereich2013}. 
Denote the degrees in the corresponding preferential attachment model by 
$\deg^{-,\gamma}_n(i)$ and consider general $f$ and associated degrees $\deg^-_n(i)$. Then, since $f\leq f^\ssup{\gamma}$ we can couple the models so that $\deg_n^{-}(i) \leq \deg^{-,\gamma}_n(i)$ for all $i\in [n], n \in \N$. In particular, we have that
\[ \E [ f(\deg^-_n(i)) ] \leq \E [f^\ssup{\gamma} (\deg^-_n(i) ) ] 
\leq \E[ f^\ssup{\gamma} (\deg^{-, \gamma}_n(i) ) ] \leq  \Big( \frac{n}{i} \Big)^\gamma, \]
as required.
\end{proof}

Using a result of~\cite{BarbourHall} for Poisson approximation (again based on the Chen-Stein method), we can now prove Theorem~\ref{mainthm3}.

\begin{proof}[Proof of Theorem~\ref{mainthm3}]
By the independence assumption for incoming edges, it follows that 
the indegree evolutions $(\deg^-_k(i))_{k \geq i}$ and $(\deg^-_k(j))_{k \geq j}$ are independent if $i \neq j$.
In particular, if we write $ X_{i,n}= \1\{ $ there is an edge from $ n$ to $ i\}
= \deg^-_n(i) - \deg^-_{n-1}(i)$, then we can write the outdegree $D_n$ of vertex $n$ as
 \[
D_n=\sum_{i=1}^{n-1}X_{i,n},
\]
i.e.\  as the sum of independent Bernoulli variables.
Note that\begin{align*}
	p_{i,n} := \P(X_{i,n}=1)
	&= \E\left[\E\left[ \deg^-_n(i) - \deg^-_{n-1}(i)\vert \mathcal{G}_{n-1}\right]\right]
	=\E\left[\frac{f(\deg_{n-1}(i))}{n-1}\right].
\end{align*}
Therefore, 
\begin{align*}
\lambda_n:=\E\left[D_n\right]
=\E\Big[\frac{1}{n-1}\sum_{i=1}^{n} f(\deg_{n-1}(i))\Big]= \E\left[f(W_{n-1})\right],
\end{align*}
where $ W_{n-1}$ denoting the indegree of a uniformly chosen vertex after the insertion of vertex $ n-1 $. From the proof of Theorem 1.1 (b) in \cite{Dereich2009} we know that
$\lambda_n \rightarrow \E[f(W)]$ if $W \sim \mu$.
Applying \cite[Thm.\ 1.1]{BarbourHall} 
we obtain that
\begin{equation}\label{eq:BH}
 d_{\rm TV}(D_n, Po(\lambda_n))\leq \frac{1-e^{-\lambda_n}}{\lambda_n} \sum_{i=1}^{n-1}p_{i,n}^2 \leq \min\{1, \frac{1}{\lambda_n}\}\sum_{i=1}^{n-1}p_{i,n}^2.
\end{equation}
In remains to control the sum on the right hand side.
By Lemma~\ref{le:moment}, we have that
\[ \sum_{i=1}^{n-1} p_{i,n}^2 = \frac{1}{(n-1)^2} \sum_{i=1}^{n-1} 
\E[ f(\deg_{n-1}^-(i))]^2 
\leq \frac{1}{(n-1)^2} \sum_{i=1}^{n-1} \Big( \frac{n}{i} \Big)^{2\gamma}  \]
Since $\la_n \ra \la:= \E[f(W)]$  we can deduce from~\eqref{eq:BH} 
that 
\[
d_{\rm TV}(D_n, Po(\lambda_n))\leq C  \, \begin{cases}
\frac{1}{n+1},&\text{ for } 0<\gamma < \frac{1}{2},\\
\frac{\log(n)}{n}, &\text{ for } \gamma = \frac{1}{2},\\
 n^{-2(1-\gamma)},&\text{ for }\frac{1}{2}< \gamma <1 ,\\
\end{cases} , 
\]
for a suitable constant ${C}> 0$, which proves the first part of Theorem~\ref{mainthm3}.

For the final part of Theorem~\ref{mainthm3}, we assume that $f(k) = \gamma k + \beta$, for $\gamma \in (0,1), \beta \in [0,1]$.
First note that in this case by~\eqref{eq_mu2}
\[\begin{aligned} \la &  = \E[f(W)] = \gamma \E[ W] + \beta = \gamma \sum_{k\geq 1} \mu ([k,\infty))  + \beta \\
& = \gamma \sum_{k \geq 1} f(k-1) \mu_{k-1} + \beta 
= \gamma \la + \beta. \end{aligned} \]
In particular, $\la = \frac{\beta}{1-\gamma}$.
Following a similar argument as in the proof of  Theorem 1.1 (b) in \cite{Dereich2009}, we have that
\[ \begin{aligned} 
&\E[ f(W_{n+1}) ]  = \frac{1}{n+1} \sum_{i=1}^{n+1} \E\Big[ \E[  f(\deg^-_{n+1}(i)) \, |\, \cG_n]\Big] \\
&= \frac{1}{n+1}\sum_{i=1}^{n} \E\Big[ \E[  f(\deg^-_{n+1}(i))-f(\deg^-_{n}(i)) \, |\, \cG_n]\Big]\\
&\hspace{0.5cm}+\frac{f(0)}{n+1}+\frac{1}{n+1} \sum_{i=1}^{n} \E\Big[ \E[  f(\deg^-_{n}(i))\, |\, \cG_n]\Big]\\
&= \frac{1}{n+1}\left( \sum_{i=1}^{n} \E\Big[ \E[  \gamma(\deg^-_{n+1}(i)-\deg^-_{n}(i)) \, |\, \cG_n]\Big]+\beta+ \sum_{i=1}^{n} \E[  f(\deg^-_{n}(i))]\right)\\
&  = \frac{1}{n+1} \sum_{i=1}^n \gamma \E\Big[ \frac{f(\deg^-_n(i)}{n}  \Big] + \frac{\beta}{n+1} +\frac{1}{n+1} \sum_{i=1}^n \E[ f(\deg^-_n(i))] \\
& = ( 1 - \frac{1- \gamma}{n+1} ) \E[ f(W_n)] + \frac{\beta}{n+1}. 
\end{aligned} \]
Using that $\la = \frac{\beta}{1-\gamma}$, we obtain that for
$\bar \la_{n+1} := \E[ f(W_n) ] - \la$, 
\[ \bar \la_{n+1} = \Big( 1- \frac{1-\gamma}{n} \Big) \bar\la_n . \]
Hence, 
\[ |\bar \la_{n+1} | = \prod_{i=1}^n \Big( 1- \frac{1-\gamma}{i} \Big) |\bar \la_1| \leq C n^{- (1-\gamma)},  \]
for a suitable constant $C>0$,
as claimed.
\end{proof}

{\bf Acknowledgements:} We would like to thank Christian D\"obler for helpful discussions. The authors gratefully acknowledge support from the DFG in form of the Research Training Group 1916 \emph{Combinatorial Structures in Geometry}.

\bibliographystyle{alpha} %{abbrv}
\bibliography{PA&Stein}

\end{document}